\documentclass[11pt,letterpaper,reqno]{amsart}
\usepackage{tikz}
\usetikzlibrary{positioning, shapes.geometric, arrows.meta, calc}
\usepackage{amssymb}
\usepackage{amsmath}
\usepackage{amsthm}
\usepackage{amsfonts}
\usepackage{bbm}
\usepackage{enumitem} 
\usepackage{pgfplots}
\pgfplotsset{compat=1.18} 
\usepackage{booktabs}
\usepackage{graphicx}
\usepackage[T1]{fontenc}
\usepackage{doi}
\usepackage{float} 
\usepackage{comment} 
\usepackage{bookmark}
\usepackage{hyperref}
\hypersetup{pdfstartview={FitH}}

\DeclareMathOperator{\rank}{rank}

\newtheorem{theorem}{Theorem}[section]
\newtheorem{lemma}[theorem]{Lemma}
\newtheorem{proposition}[theorem]{Proposition}
\newtheorem{corollary}[theorem]{Corollary}

\newtheorem{problem}[theorem]{Problem}

\theoremstyle{definition}

\newtheorem{remark}[theorem]{Remark}

\let\originalleft\left
\let\originalright\right
\renewcommand{\left}{\mathopen{}\mathclose\bgroup\originalleft}
\renewcommand{\right}{\aftergroup\egroup\originalright}

\newcommand{\E}{\mathbb{E}}        
\newcommand{\Var}{\mathrm{Var}}    

\newcommand{\PP}{\mathbb{P}}     
\newcommand{\Q}{\mathbb{Q}}
\newcommand{\ind}{\mathbf{1}}      
\newcommand{\F}{\mathbb{F}}

\begin{document}

\title{An Erd\H{o}s problem on random subset sums in finite abelian groups}

\author[J.~Ma]{Jie Ma}
\author[Q.~Tang]{Quanyu Tang}

\address{School of Mathematical Sciences, University of Science and Technology of China, Hefei, Anhui 230026, and Yau Mathematical Sciences Center, Tsinghua University, Beijing 100084, China}
\email{jiema@ustc.edu.cn}
\address{School of Mathematics and Statistics, Xi'an Jiaotong University, Xi'an 710049, P. R. China}
\email{tang\_quanyu@163.com}



\begin{abstract}
Let $f(N)$ denote the least integer $k$ such that, if $G$ is an abelian group of order $N$ and $A \subseteq G$ is a uniformly random $k$-element subset, then with probability at least $\tfrac12$ the subset-sum set $\{ \sum_{x \in S} x : S \subseteq A \}$ equals $G$.
In 1965, Erd\H{o}s and R\'{e}nyi proved that for all $N$,
$$
f(N) \le \log_2 N + \left(\frac{1}{\log 2}+o(1)\right)\log\log N.
$$
Erd\H{o}s later conjectured that this bound cannot be improved to $f(N)\le \log_2 N+o(\log\log N)$.
In this paper we confirm this conjecture by showing that, for primes $p$,
$$
f(p)\ge \log_2 p+\left(\frac{1}{2\log 2}+o(1)\right)\log\log p.
$$
\end{abstract}

\maketitle

\section{Introduction}

The study of random subset sums in finite abelian groups dates back to the work of Erd\H{o}s and R\'{e}nyi~\cite{ErRe65} and of Erd\H{o}s and Hall~\cite{ErHa78b}. Let $f(N)$ denote the least integer $k$ such that the following holds: if $G$ is an abelian group of order $N$ and
$A\subseteq G$ is a uniformly random $k$-element subset, then with probability at least $1/2$ every element of $G$
can be written as a sum of distinct elements of $A$. Erd\H{o}s and R\'{e}nyi~\cite{ErRe65} proved the universal upper bound
\begin{equation}\label{eq:erre_upper_bound_v1}
f(N) \le \log_2 N + O(\log\log N),
\end{equation}
whereas Erd\H{o}s and Hall~\cite{ErHa78b} showed that this cannot, in general, be improved to
\[
f(N)\le \log_2 N + o(\log\log\log N).
\]
In 1973, Erd\H{o}s~\cite[p.~127]{Er73} conjectured that the $O(\log\log N)$ error term in~\eqref{eq:erre_upper_bound_v1}
cannot be sharpened to $o(\log\log N)$.
This conjecture is listed as Problem~\#543 on Bloom's Erd\H{o}s Problems website~\cite{EP543} in the following self-contained form.

\begin{problem}\label{prob:EP543}
Define $f(N)$ to be the minimal $k$ such that the following holds: if $G$ is an abelian group of size $N$ and
$A\subseteq G$ is a uniformly random $k$-element subset, then with probability at least $1/2$ every element of $G$
can be written in the form $\sum_{x\in S}x$ for some $S\subseteq A$. Is
\[
f(N) \le \log_2 N + o(\log\log N)\,?
\]
\end{problem}

In this paper we confirm Erd\H{o}s' conjecture by answering Problem~\ref{prob:EP543} in the negative.
We restrict attention to prime orders $N=p$ (so $G$ has to be isomorphic to the cyclic group $\F_p$ of order $p$). This suffices to rule out the proposed bound
$f(N)\le \log_2 N+o(\log\log N)$ for general $N$.
The notation $\F_p$ usually stands for the field of $p$ elements, and here we only view it as an additive group. 
For any subset $A\subseteq \F_p$ we write
\[
\Sigma(A):=\Bigl\{\sum_{x\in S}x:\ S\subseteq A\Bigr\},
\]
where the empty sum is set to be $0$.
In this setting, equivalently, $f(p)$ denotes the least integer $k$ such that a uniformly random $k$-element subset
$A\subseteq \F_p$ satisfies $\PP(\Sigma(A)=\F_p)\ge \tfrac12$.
Our main result is as follows.

\begin{theorem}\label{thm:counterexample}
Fix any constant $c$ with $0<c<\frac{1}{2\log 2}$. Let $p$ be prime and let $A\subseteq \mathbb{F}_p$ be a uniformly random $k$-element subset with $k=\lfloor \log_2 p + c\log\log p\rfloor$. Then
\[
\PP\bigl(\Sigma(A)=\mathbb{F}_p\bigr)\longrightarrow 0 \qquad (p\to\infty).
\]
\end{theorem}

This provides a quantitative bound for primes, which answers Problem~\ref{prob:EP543} in the negative.

\begin{corollary}\label{cor:fp-lower}
For all sufficiently large primes $p$, it holds that
\[
f(p)\ge \log_2 p + \left(\frac{1}{2\log 2}+o(1)\right)\log\log p.
\]
\end{corollary}

\subsection{Proof overview and comparison with earlier work}

The Erd\H{o}s--R\'enyi upper bound~\cite{ErRe65} sets the benchmark for the problem. We now discuss the lower-bound side and the main ideas of the present proof.

A closely related earlier work is that of Erd\H{o}s and Hall~\cite{ErHa78b}, who worked in an i.i.d.\ model on a finite abelian group \(G\) of order \(n\). For every \(r\geq 0\), they defined \(d(r)\) to be the number of group elements having exactly \(r\) representations as subset sums of \(k\) independent random elements. In particular, \(d(0)\) denotes the number of \emph{missed} elements (i.e., elements not representable as subset sums of these random elements). Write \(\mu=2^k/n\). They proved a general theorem showing that, under the assumption that the number of elements of each fixed order in \(G\) is \(o(n)\), if \(k=\log_2 n+O(1)\), then  for every fixed integer \(r\ge 0\),
\[
d(r)=\bigl(1+o(1)\bigr)\,n e^{-\mu}\frac{\mu^r}{r!}
\]
with probability tending to \(1\) as \(n\to\infty\). Their proof is moment-theoretic, based on the analysis of fixed-order moments and variances of the representation counts, combined with character-sum estimates.

For the present problem, however, the decisive quantity is the single statistic \(d(0)\). Erd\H{o}s and Hall~\cite{ErHa78b} also proved an additional result for cyclic groups: there exists an absolute constant \(b>0\) such that \(d(0)>0\) with probability tending to \(1\) whenever \(\mu<b\log\log n\). Since \(\mu=2^k/n\), this is equivalent to
\[
k<\log_2 n+\log_2(b\log\log n)
=\log_2 n+O(\log\log\log n).
\]
Their proof again relies on the moment method, but uses the additional structure of cyclic groups to extend the fixed-moment analysis to a growing range of moment orders. For broader background on sets of multiples and probabilistic group theory, see Hall~\cite[Chapter~4]{Hall96}.

The present paper is close in spirit to~\cite{ErHa78b}, but technically quite different. We work only in the prime cyclic case \(G=\F_p\), and we do not seek asymptotics for all fixed multiplicities \(d(r)\). Instead, after reducing the random \(k\)-subset model to an i.i.d.\ model, we focus on the number \(d(0)\) of missed elements and reduce the problem to estimating the probability $\PP_B$ that all elements of a prescribed subset \(B\subseteq \F_p\setminus\{0\}\) with \(|B|=1\) or \(2\) are missed. 
More precisely, if \(X_B\) denotes the number of indexed subset sums landing in
\(B\), then we require a Poisson-type estimate for \(X_B\) with mean $|B|\lambda$, where
\[
\lambda=\frac{2^k-1}{p}\to\infty.
\]
In our application we take \(k=\lfloor \log_2 p+c\log\log p\rfloor\), which implies
\[
\lambda\asymp (\log p)^\alpha
\qquad \alpha=c\log 2\in(0,1/2).
\]
We then estimate the desired probability \(\PP_B:=\PP(X_B=0)\) using Bonferroni-type
inequalities, which require uniform control of the factorial moments of
\(X_B\) up to a truncated level \(R\asymp (\log p)^{O(1)}\to\infty\); see the proof of Proposition~\ref{prop:poisson-miss}.
This approach is technically different from that of
Erd\H{o}s--Hall~\cite{ErHa78b}.

Another new technical ingredient is a quantitative bound for low-rank
incidence matrices (see Lemma~\ref{lem:fullrank-dominates}), obtained via a Boolean-cube intersection bound.
Together, these ingredients allow us to raise the earlier cyclic obstruction
from the \(\log\log\log p\)-scale to the \(\log\log p\)-scale.

\subsection{Paper organization}
The rest of the paper is organized as follows. Section~\ref{sec:preliminaries} collects notation and several auxiliary lemmas. In Section~\ref{sec:proof} we prove Theorem~\ref{thm:counterexample} and Corollary~\ref{cor:fp-lower} by first establishing the key Poisson-like estimate, Proposition~\ref{prop:poisson-miss}, and then combining it with a second-moment argument. Section~\ref{sec:concluding} contains several concluding remarks.

\section{Preliminaries}\label{sec:preliminaries}

\subsection{Notation}
Throughout, $\log$ denotes the natural logarithm and $\log_2 x:=\log x/\log 2$. For a positive integer $k$, we define $[k] := \{1, \dots, k\}$. 
For $Y\in\mathbb{R}$ and $r\in\mathbb{N}$, let $(Y)_r:=Y(Y-1)\cdots(Y-r+1)$ denote the falling factorial, with the convention $(Y)_0:=1$. 

We use Vinogradov's asymptotic notation. For functions \(f=f(n)\) and \(g=g(n)\), we write
\(f=O(g)\), \(g=\Omega(f)\), \(f\ll g\), or \(g\gg f\) to mean that there exists a constant \(C>0\) such that
\(|f(n)|\le C g(n)\) for all sufficiently large \(n\).
We write \(f\asymp g\) or \(f=\Theta(g)\) to mean that \(f\ll g\) and \(g\ll f\), and we write
\(f=o(g)\) to mean that \(f(n)/g(n)\to 0\) as \(n\to\infty\).

\subsection{Auxiliary lemmas}
We now present several auxiliary lemmas that will be used repeatedly in the proof of the main theorem. We begin with two linear-algebraic observations for $\{0,1\}$-matrices.

Let \(\mathbb Q\) denote the field of rational numbers. We equip \(\mathbb Q^r\) with the standard inner product, and write \(W^\perp\) for the orthogonal complement of a subspace \(W\subseteq \mathbb Q^r\).

\begin{lemma}\label{lem:rank-stability}
Fix a constant $\beta$ with $0<\beta<\tfrac12$. Let \(r,k\ge 1\), and let \(V\in\{0,1\}^{r\times k}\) with $r\le (\log p)^{\beta}$. Then for all sufficiently large primes $p$,
\[
\rank_{\Q}(V)=\rank_{\F_p}(V).
\]
\end{lemma}

\begin{proof}
Let $s=\rank_{\Q}(V)$. 
Then there exists some $s\times s$ minor of $V$ with nonzero integer determinant $\Delta$.
By Hadamard's inequality,
\[
|\Delta|\le s^{s/2}\le r^{r/2}.
\]
Since $r\le (\log p)^{\beta}$, we have
\[
\log\bigl(r^{r/2}\bigr)=\frac r2\log r
\le \frac12(\log p)^{\beta}\log\log p
=o(\log p),
\]
because $\beta<1$. Hence $r^{r/2}<p$ for all sufficiently large $p$, so $\Delta\not\equiv 0\pmod p$ and therefore
$\rank_{\F_p}(V)\ge s$.
The reverse inequality $\rank_{\F_p}(V)\le \rank_{\Q}(V)$ holds for any integer matrix, since reducing modulo $p$
cannot increase rank. Thus $\rank_{\F_p}(V)=\rank_{\Q}(V)$.
\end{proof}

\begin{lemma}\label{lem:no-sparse}
Let \(r,k\ge 1\), let \(V\in\{0,1\}^{r\times k}\) have distinct and nonzero rows, and let \(W\) denote the column space of \(V\) over \(\mathbb{Q}\). Then \(W^\perp\) contains no nonzero vector supported on at most two coordinates.
\end{lemma}

\begin{proof}
Let \(0\ne y\in W^\perp\), and write \(I=\operatorname{supp}(y)\subseteq [r]\).
Then for every \(x\in W\),
\[
\sum_{i\in I} y_i x_i = 0.
\]
In particular, applying this to each column \(v^{(t)}\in W\) of \(V\), we get
\[
\sum_{i\in I} y_i V_{i,t}=0
\qquad (1\le t\le k).
\]
Hence \(\sum_{i\in I} y_i\,R_i=0\), where \(R_i\in\mathbb Q^k\) denotes the \(i\)th row of \(V\). Thus the rows of
\(V\) indexed by \(I\) are linearly dependent over \(\mathbb Q\).

If \(|I|=1\), this says that one row of \(V\) is zero, contradicting the
hypothesis. If \(|I|=2\), then two rows of \(V\) are proportional. Since these
rows are nonzero vectors in \(\{0,1\}^k\), proportionality forces them to be
equal, again contradicting the hypothesis that the rows are distinct.
\end{proof}

Next, we provide a simple bound on the size of $W\cap\{0,1\}^r$
for proper subspaces $W\subseteq\Q^r$ under a mild non-degeneracy condition (i.e., the conclusion of Lemma~\ref{lem:no-sparse}).

\begin{lemma}\label{lem:34}
Let $W\subseteq\mathbb Q^r$ be a $d$-dimensional subspace with $d<r$. If $W^\perp$ contains no nonzero vector supported
on at most two coordinates, then
\[
|W\cap\{0,1\}^r|\le \frac34\,2^d.
\]
\end{lemma}

\begin{proof}
Choose a coordinate set $J\subseteq[r]$ with $|J|=d$ such that the projection $\pi_J:W\to\mathbb Q^J$ is an isomorphism.
Thus each $x\in W$ is uniquely determined by $u:=x|_J\in\mathbb Q^d$, and for each $i\notin J$ there exists a homogeneous linear form
$L_i:\mathbb Q^d\to\mathbb Q$ with $x_i=L_i(u)$. For $j\in J$, also set $L_j(u)=u_j$. Hence
\[
|W\cap\{0,1\}^r|
=
\left|
\left\{
u\in\{0,1\}^d:\ L_i(u)\in\{0,1\}\ \forall i\in [r]\setminus J
\right\}
\right|.
\]

The assumption on $W^\perp$ implies that the forms $L_1,\dots,L_r$ are pairwise linearly independent. Indeed, if
$L_i\equiv 0$, then $x_i=0$ for all $x\in W$, so $e_i\in W^\perp$, where $e_i$ denotes the $i$th standard basis vector of $\mathbb Q^r$; and if $L_i=cL_j$ for some distinct $i,j\in [r]$
and some $c\in\mathbb Q^\times$, then $x_i-cx_j=0$ for all $x\in W$, so $e_i-ce_j\in W^\perp$. In particular, for every
$i\in [r]\setminus J$, the form $L_i$ has at least two nonzero coefficients.

Fix any $i_0\in [r]\setminus J$ and write $L_{i_0}(u)=\sum_{\ell=1}^d \alpha_\ell u_\ell$ with at least two $\alpha_\ell\ne0$.
Let $Z$ be uniform on $\{0,1\}^d$. Conditioning on all coordinates except two indices $s\ne t$ with $\alpha_s,\alpha_t\ne0$,
the random variable $L_{i_0}(Z)$ becomes $C+aX+bY$ where $X,Y$ are independent $\mathrm{Bernoulli}(1/2)$ and $a,b\ne0$.
Thus $C+aX+bY$ takes four equiprobable values; in the degenerate case $a=\pm b$ it takes three values with masses
$1/4,1/2,1/4$. In either case,
\[
\mathbb P\bigl(L_{i_0}(Z)\in\{0,1\}\bigr)\le \frac34.
\]
Therefore at most a $3/4$ fraction of $u\in\{0,1\}^d$ satisfy $L_{i_0}(u)\in\{0,1\}$, and imposing further constraints
can only reduce the count. Hence $|W\cap\{0,1\}^r|\le \frac34\,2^d$.
\end{proof}

For the remainder of the paper, let $p$ be a sufficiently large prime and set
\begin{equation}\label{eq:def-k-M-lambda}
k:=\Bigl\lfloor \log_2 p + c\log\log p\Bigr\rfloor,\qquad
M:=2^k-1,\qquad \mbox{ and } \qquad
\lambda:=\frac{M}{p},
\end{equation}
where $c\in \big(0,\frac{1}{2\log 2}\big)$ is arbitrary fixed.
We will use the following estimate repeatedly.

\begin{lemma}\label{lem:lambda}
Let $\alpha:=c\log 2\in(0,1/2)$. 
Then $\lambda=\Theta((\log p)^{\alpha})$. In particular,
\[
\lambda=o(\log p)\qquad\text{and}\qquad e^{\lambda}=p^{o(1)}.
\]
\end{lemma}
\begin{proof}
By \eqref{eq:def-k-M-lambda}, we have
\[
2^k \asymp p(\log p)^{c\log 2}=p(\log p)^\alpha.
\]
Since \(M=2^k-1\sim 2^k\), we get \(\lambda=M/p=\Theta((\log p)^\alpha)\).
As \(\alpha<1\), this gives \(\lambda=o(\log p)\), and hence \(e^\lambda=p^{o(1)}\).
\end{proof}

We will also need the following crude but convenient estimate for falling factorials.

\begin{lemma}\label{lem:factorial-mr}
Fix a constant $\beta$ with $0<\beta<\tfrac12$. Assume $m\in\{1,2\}$ and $r\le (\log p)^{\beta}$. Then
\[
\frac{(mM)_r}{p^r}=m^r\,p^{o(1)}.
\]
\end{lemma}

\begin{proof}
By Lemma~\ref{lem:lambda},
\[
M\asymp p(\log p)^\alpha.
\]
Since \(r\le (\log p)^\beta\) with \(\beta<1/2\), we have \(r^2/M=o(1)\). Therefore
\begin{equation}\label{eq:factorial_mr_v1}
(mM)_r
=(mM)^r\prod_{j=0}^{r-1}\Bigl(1-\frac{j}{mM}\Bigr)
=(mM)^r\exp\Bigl(O\Bigl(\frac{r^2}{M}\Bigr)\Bigr)
=(mM)^r(1+o(1)),
\end{equation}
uniformly for \(m\in\{1,2\}\). Dividing by \(p^r\), we get
\[
\frac{(mM)_r}{p^r}
=
m^r\Bigl(\frac{M}{p}\Bigr)^r(1+o(1))
=
m^r\lambda^r(1+o(1)).
\]
Finally, by Lemma~\ref{lem:lambda},
\[
\log(\lambda^r)=r\log\lambda=O((\log p)^\beta\log\log p)=o(\log p),
\]
so \(\lambda^r=p^{o(1)}\). Hence
\[
\frac{(mM)_r}{p^r}=m^r\,p^{o(1)}.
\qedhere\]
\end{proof}
\begin{remark}
The hypotheses in Lemmas~\ref{lem:rank-stability} and~\ref{lem:factorial-mr} can be weakened. We have chosen not to optimize them here, since the stated forms already suffice for all subsequent applications in this paper.
\end{remark}

\section{Proof of Main Result}\label{sec:proof}
In this section we prove Theorem~\ref{thm:counterexample} and Corollary~\ref{cor:fp-lower} in three stages. 
First, in Section~\ref{sec:reduce_iid}, we reduce the problem to an independent (i.i.d.) model. 
The proof of the theorem itself (Section~\ref{sec:subsec3.2}) is based on a second-moment analysis, which relies on a crucial Poisson-like estimate (Proposition~\ref{prop:poisson-miss}). 
The proof of this proposition, which constitutes our main technical contribution, is given in Section~\ref{sec:subsec3.3} and proceeds by applying Bonferroni inequalities as well as estimating factorial moments of an associated Poisson-like random variable.

\subsection{Reducing to an independent model}\label{sec:reduce_iid}
Fix a prime $p$ and an integer $k=O(\log p)$.
We compare the following two distributions on subsets of $\F_p$.
In the \emph{subset model}, $A_{\mathrm{subset}}$ is chosen uniformly from $\binom{\F_p}{k}$, and for any event $F$ that depends only on the subset $A_{\mathrm{subset}}$, we write
\[
\PP_{\mathrm{subset}}(F):=\PP\bigl(F(A_{\mathrm{subset}})\bigr).
\]
In the \emph{independent (i.i.d.)\ model}, let \(a_1,\dots,a_k\) be i.i.d.\ uniform on \(\F_p\), and set
\[
A_{\mathrm{iid}}:=(a_1,\dots,a_k).
\]
Thus \(A_{\mathrm{iid}}\) is an ordered \(k\)-tuple; when convenient, we also regard it as a multiset, so repetitions are allowed. For the i.i.d.\ sample \(A_{\mathrm{iid}}=(a_1,\dots,a_k)\), we define
\[
\Sigma(A_{\mathrm{iid}}) := \Bigl\{\sum_{i\in I} a_i:\ I\subseteq [k] \Bigr\}.
\]
This depends only on the underlying multiset of \(a_1,\dots,a_k\). For any event \(F\) depending only on the underlying multiset \(A_{\mathrm{iid}}\), we write
\[
\PP_{\mathrm{iid}}(F):=\PP\bigl(F(A_{\mathrm{iid}})\bigr).
\]

Let $E$ be the event that $a_1,\dots,a_k$ are pairwise distinct. Since $k=O(\log p)$, a union bound gives
\begin{equation}\label{eq:union_bound_1}
\PP(E^c)=\PP(\exists\, i<j:\ a_i=a_j)\le \binom{k}{2}\frac{1}{p}=o(1)\qquad (p\to\infty).
\end{equation}
Conditioned on $E$, the multiset $A_{\mathrm{iid}}$ is uniformly distributed over $\binom{\F_p}{k}$,
so $A_{\mathrm{iid}}\mid E$ has the same distribution as $A_{\mathrm{subset}}$. Let \(F\) be any event in the i.i.d.\ model that depends only on the multiset \(A_{\mathrm{iid}}=\{a_1,\dots,a_k\}\). Then the corresponding event in the subset model satisfies
\[
\PP_{\mathrm{subset}}(F)=\PP_{\mathrm{iid}}(F\mid E).
\]
By the law of total probability, for any such event $F$ we have
\[
\bigl|\PP(F\mid E)-\PP(F)\bigr|\le 2\PP(E^c).
\]Consequently, for every such event $F$, by~\eqref{eq:union_bound_1} we have
\begin{equation}\label{eq:coupling}
\bigl|\PP_{\mathrm{subset}}(F)-\PP_{\mathrm{iid}}(F)\bigr|
=\bigl|\PP(F(A_{\mathrm{iid}})\mid E)-\PP(F(A_{\mathrm{iid}}))\bigr|
\le 2\,\PP(E^c)=o(1).
\end{equation}
Applying \eqref{eq:coupling} to the event $F=\{\Sigma(A_{\mathrm{iid}})=\F_p\}$ yields
\begin{equation}\label{eq:coupling_vv2}
\PP_{\mathrm{subset}}\bigl(\Sigma(A)=\F_p\bigr)
\le \PP_{\mathrm{iid}}\bigl(\Sigma(A_{\mathrm{iid}})=\F_p\bigr)+o(1).
\end{equation}
Therefore, in order to prove Theorem~\ref{thm:counterexample} it suffices to show that
\[
\PP_{\mathrm{iid}}\bigl(\Sigma(A_{\mathrm{iid}})=\F_p\bigr)\longrightarrow 0 \qquad (p\to\infty).
\]
Henceforth all probabilities and expectations are taken in the i.i.d.\ model, and for convenience, we omit the subscript ``iid'' throughout the rest of the paper.

\subsection{Proof of Theorem~\ref{thm:counterexample} via a Poisson-like estimate}\label{sec:subsec3.2}

To prove Theorem~\ref{thm:counterexample}, we rephrase the event $F=\{\Sigma(A)=\F_p\}$ in terms of the number of ``missed'' elements in $\F_p$ as defined below. 
For each nonempty $S\subseteq [k]$, define the {\it indexed subset sum}
\begin{equation}\label{eq:sigmaS}
\sigma(S):=\sum_{i\in S} a_i\in \F_p. 
\end{equation}
For each $x\in \F_p^\times:=\F_p\setminus \{0\}$, let
\[
X_x:=\#\{\varnothing\ne S\subseteq [k]:\ \sigma(S)=x\}.
\]
Thus $X_x=0$ denotes the event that $x$ is missed by all nonempty indexed subset sums. Set
\[
U:=\#\{x\in \F_p^\times:\ X_x=0\}=\sum_{x\in\F_p^\times}\ind_{\{X_x=0\}}.
\]
Since $0\in\Sigma(A)$ always (via the empty sum), we have
$\Sigma(A)=\F_p\ \Longrightarrow\ U=0$,
and consequently
\begin{equation}\label{eq:reduce-U}
\PP\bigl(\Sigma(A)=\F_p\bigr)\le \PP(U=0).
\end{equation}

The technical core of this section is as follows.
For $B\subseteq\F_p^\times$, define $X_B:=\sum_{x\in B}X_x$. 
Intuitively speaking, one can view each $X_x$ as a Poisson random variable of mean $\lambda=\frac{M}{p}=\frac{2^k-1}{p}$, where each non-empty subset $S\subseteq [k]$ contributes the value $\sum_{i\in S} a_i$ uniformly in $\F_p$.

\begin{proposition}\label{prop:poisson-miss}
Fix $m\in\{1,2\}$. For any $B\subseteq\F_p^\times$ with $|B|=m$,
\[
\PP\bigl(X_B=0\bigr)=(1+o(1)) e^{-m\lambda}\qquad (p\to\infty),
\]
where the $o(1)$ is uniform over all such $B$.
\end{proposition}

Assuming Proposition~\ref{prop:poisson-miss}, we now complete the proofs of Theorem~\ref{thm:counterexample} and Corollary~\ref{cor:fp-lower}.

\begin{proof}[Proof of Theorem~\ref{thm:counterexample}, assuming Proposition~\ref{prop:poisson-miss}]
In this proof, we use a standard second-moment argument; see, for example, \cite{AlSp16}.
By Proposition~\ref{prop:poisson-miss} with $m=1$, uniformly for all
\(x\in\F_p^\times\),
\[
\PP(X_x=0)=(1+o(1)) e^{-\lambda}.
\]
Recall that \(U=\sum_{x\in\F_p^\times}\ind_{\{X_x=0\}}\). By linearity of expectation, we have
\[
\E U=\sum_{x\in\mathbb F_p^\times}\mathbb P(X_x=0)=(1+o(1))(p-1)e^{-\lambda}.
\]
Similarly, for distinct $x\ne y$ in $\F_p^\times$, Proposition~\ref{prop:poisson-miss} (applied with $m=2$ and $B=\{x,y\}$) yields, uniformly in such pairs,
\[
\PP(X_x=0,\ X_y=0)=(1+o(1)) e^{-2\lambda},
\]
and hence
\begin{align*}
\E U^2
&=
\sum_{x\in\F_p^\times}\PP(X_x=0)
+
\sum_{\substack{x,y\in\F_p^\times\\x\neq y}}\PP(X_x=0,\ X_y=0)
\\&=
(1+o(1))(p-1)e^{-\lambda}+(1+o(1))(p-1)(p-2)e^{-2\lambda}.
\end{align*}
Hence
\[
\Var(U)=\E U^2-(\E U)^2=O(pe^{-\lambda})+o(p^2e^{-2\lambda}).
\]
By Lemma~\ref{lem:lambda}, \(\lambda=o(\log p)\), so
\[
\E U\asymp pe^{-\lambda}\to\infty.
\]
Therefore
\[
\Var(U)=O(\E U)+o\bigl((\E U)^2\bigr)=o\bigl((\E U)^2\bigr).
\]
By Chebyshev's inequality,
\[
\PP(U=0)\le \frac{\Var(U)}{(\E U)^2}=o(1).
\]
Combining this with \eqref{eq:reduce-U} and \eqref{eq:coupling_vv2}, we obtain
\[
\PP_{\mathrm{subset}}\bigl(\Sigma(A)=\F_p\bigr)=o(1),
\]
as required.
\end{proof}

\begin{proof}[Proof of Corollary~\ref{cor:fp-lower}, assuming Proposition~\ref{prop:poisson-miss}]
Let $\varepsilon>0$ be an arbitrary constant, and set \(c=\frac{1}{2\log 2}-\varepsilon\). By Theorem~\ref{thm:counterexample}, if \(k=\lfloor \log_2 p + c\log\log p\rfloor\), then
\[
\PP_{\mathrm{subset}}\bigl(\Sigma(A)=\F_p\bigr)\to 0
\qquad (p\to\infty).
\]
In particular, for all sufficiently large primes $p$, this probability is less than $\tfrac12$. By the definition of $f(p)$, it follows that \(f(p)>k\). Therefore
\[
f(p)\ge k+1
\ge \log_2 p+\left(\frac{1}{2\log 2}-\varepsilon\right)\log\log p,
\]
and the result follows.
\end{proof}

The remainder of this section is devoted to the proof of Proposition~\ref{prop:poisson-miss}.

\subsection{Proof of Proposition~\ref{prop:poisson-miss}}\label{sec:subsec3.3}
Fix $m\in\{1,2\}$ and $B\subseteq \F_p^\times$ with $|B|=m$. 
We estimate the factorial moments of $X_B$ by comparing them with those of a Poisson random variable of mean $m\lambda$,
and then apply these bounds in the following Bonferroni-type inequalities (which arise from the inclusion-exclusion principle on factorial moments; see, e.g.,~\cite[Eq.~(1.276)]{JKK05} and \cite[Exercises~1.6.9--1.6.10]{Durrett19}). 
For every nonnegative integer-valued random variable $Z$ with finite support, it holds that
\[
\PP(Z=0)=\sum_{r=0}^{\infty}(-1)^r\frac{\E[(Z)_r]}{r!},
\]
and for every integer $R\ge 0$,
\begin{equation}\label{eq:bonferroni}
\left|\PP(Z=0)-\sum_{r=0}^{R}(-1)^r\frac{\E[(Z)_r]}{r!}\right|
\le \frac{\E[(Z)_{R+1}]}{(R+1)!}.
\end{equation}

Throughout the rest of this proof, choose an exponent $\beta$ with
\[
\alpha<\beta<\frac12,
\]
where $\alpha=c\log 2$ is the same as in Lemma~\ref{lem:lambda}, and set
\[
R:=\Bigl\lfloor (\log p)^{\beta}\Bigr\rfloor.
\]
Recall that $M=2^k-1$ and $\lambda=M/p$. Since $\lambda=o(R)$, it is enough for the Bonferroni step to prove that there exists a constant $\eta_*>0$\footnote{Later in the proof, we will rename this constant as $\eta_4$ once all error terms have been assembled.} such that
\begin{equation}\label{eq:goal-moment}
\mathbb E[(X_B)_r]=(m\lambda)^r\bigl(1+O(p^{-\eta_*})\bigr)
\qquad (1\le r\le R),
\end{equation}
uniformly for all $B\subseteq\F_p^\times$ with $|B|=m$. As we will show later, once this estimate is proved, substituting it into \eqref{eq:bonferroni} with $Z=X_B$ completes the proof.

We next introduce some notation.
Write $a=(a_1,\dots,a_k)\in\F_p^k$.  For each nonempty $S\subseteq[k]$, let $v(S)\in\{0,1\}^k$ be its indicator vector. So the definition \eqref{eq:sigmaS} becomes $\sigma(S)=a\cdot v(S)$. 
Given subsets $S_1,\dots,S_r\subseteq[k]$, we write
$V=V(S_1,\dots,S_r)\in\{0,1\}^{r\times k}$ for the incidence matrix with entries $V_{j,t}=\ind_{\{t\in S_j\}}$. 

The following lemma relates estimates of certain probabilities to corresponding ranks.
Fix distinct nonempty subsets $S_1,\dots,S_r$ and elements $b_1,\dots,b_r\in\F_p$, and let $d:=\rank_{\F_p}(V)$.

\begin{lemma}\label{lem:rank-prob}
With notation as above, the system $a\cdot v(S_j)=b_j$ for $1\le j\le r$ is either inconsistent (probability $0$),
or consistent and then
\[
\PP\bigl(\sigma(S_j)=b_j\ \forall j\bigr)=p^{-d}.
\]
\end{lemma}
\begin{proof}
Define the linear map $T:\mathbb F_p^k\to\mathbb F_p^r$ by
$T(a)=(a\cdot v(S_1),\dots,a\cdot v(S_r))$.
Then $\dim(\mathrm{Im}\,T)=\mathrm{rank}_{\mathbb F_p}(V)=d$, so $|\mathrm{Im}\,T|=p^d$ and every fiber has size $|\ker T|=p^{k-d}$.
If $b:=(b_1,\dots,b_r)\notin\mathrm{Im}\,T$ then $\mathbb P(T(a)=b)=0$;
otherwise $\mathbb P(T(a)=b)=p^{k-d}/p^k=p^{-d}$.
\end{proof}

Now we expand $(X_B)_r$ by counting ordered $r$-tuples of distinct pairs $(S_j,b_j)$ with $\varnothing\ne S_j\subseteq[k]$ and $b_j\in B$:
\[
(X_B)_r=\sum_{\substack{(S_1,b_1),\dots,(S_r,b_r)\\ \text{all distinct}}}
\ind_{\{\sigma(S_1)=b_1,\dots,\sigma(S_r)=b_r\}}=
\sum_{\substack{S_1,\dots,S_r\subseteq [k]\\
\varnothing\ne S_i,\ \text{all distinct}}}
\ \sum_{b_1,\dots,b_r\in B}
\ind_{\{\sigma(S_1)=b_1,\dots,\sigma(S_r)=b_r\}}.
\]
Here the second equality is valid because if \(S_i=S_j\) and \(b_i\ne b_j\),
then the corresponding indicator is identically zero, while if \(S_i=S_j\) and
\(b_i=b_j\), the pairs \((S_i,b_i)\) and \((S_j,b_j)\) are not distinct.

By taking expectations and applying Lemma~\ref{lem:rank-prob}, we have
\begin{equation}\label{eq:moment-rank}
\E[(X_B)_r]=\sum_{d=1}^{r}\frac{N_{r,d}(B)}{p^d},
\end{equation}
where \(N_{r,d}(B)\) denotes the number of ordered tuples
\[
(S_1,\dots,S_r;\,b_1,\dots,b_r),
\]
with \(S_1,\dots,S_r\) distinct nonempty subsets of \([k]\) and \(b_1,\dots,b_r\in B\), for which the incidence matrix has rank \(d\) over \(\F_p\) and the corresponding system is consistent. By Lemma~\ref{lem:rank-stability}, we may freely interpret ranks over \(\Q\).

The main term in the expression \eqref{eq:moment-rank} of $\E[(X_B)_r]$ comes from $d=r$. To see this, we need a crude bound on the number of low-rank incidence ``patterns'' as follows.

\begin{lemma}\label{lem:lowrankcount}
Fix $r\le R$ and $d<r$.
Let $T_{r,d}$ be the number of ordered $r$-tuples of distinct nonempty subsets $(S_1,\dots,S_r)$ of $[k]$ whose incidence matrix has rank $d$ over $\mathbb{Q}$.
Then
\[
T_{r,d} \le 2^{r^2}\Bigl(\tfrac34\,2^d\Bigr)^k.
\]
\end{lemma}

\begin{proof}
Let $V$ be an incidence matrix of rank $d < r$ and let \(W\) denote the column space of \(V\) over \(\mathbb{Q}\). The number of possible $W$ is at most the number of ordered $d$-tuples of vectors in $\{0,1\}^r$ spanning $W$, hence at most $(2^r)^d\le 2^{r^2}$. Now fix $W$. Every column of $V$ must lie in $\Gamma:=W\cap\{0,1\}^r$. By Lemma~\ref{lem:no-sparse} and Lemma~\ref{lem:34}, we have $|\Gamma|\le \frac34\,2^d$. Thus the number of possible column sequences is at most $|\Gamma|^k$, giving the stated bound.
\end{proof}

We now establish~\eqref{eq:goal-moment}. Fix $1\le r\le R$.
Using \eqref{eq:moment-rank} we express $\E[(X_B)_r]$ as a sum of contributions from the full-rank case $d = r$ and the low-rank cases $1 \leq d < r$.

We first bound the low-rank cases.
Let $d<r$.
There are at most $m^r$ ways to select elements $b_1, ..., b_r\in B$, so we have $N_{r,d}(B)\le m^r T_{r,d}$. 
By Lemma~\ref{lem:lowrankcount}, we have
\[
\sum_{d<r}\frac{N_{r,d}(B)}{p^d}
\le m^r\,2^{r^2}\Bigl(\tfrac34\Bigr)^k\sum_{d<r}\left(\frac{2^k}{p}\right)^d.
\]
Because $r\le R\le(\log p)^\beta$ with $\beta<1/2$, we have $2^{r^2}=p^{o(1)}$.
Moreover, by Lemma~\ref{lem:lambda}, 
we have $2^k/p\leq 2\lambda \ll(\log p)^\alpha$, so it follows that 
\[
\max_{d<r}\left(\frac{2^k}{p}\right)^d \le (\log p)^{\alpha r}
=\exp\bigl(O(r\log\log p)\bigr)=p^{o(1)}.
\]
Finally, we have $(3/4)^k=p^{-\eta_0+o(1)}$ with $\eta_0:=-\log_2(3/4)>0$.
Hence, by Lemma~\ref{lem:factorial-mr}, there exists a fixed $\eta>0$ such that uniformly for all $r\le R$ and all $|B|=m\in\{1,2\}$,
\begin{equation}\label{eq:lowrank-negl}
\sum_{d<r}\frac{N_{r,d}(B)}{p^d}
\le r m^r p^{-\eta_0+o(1)}
\le \frac{(mM)_r}{p^r}\,p^{-\eta}.
\end{equation}

For the full-rank case, note that when \(\rank(V)=r\), the map \(T\) (defined in the proof of Lemma~\ref{lem:rank-prob}) is surjective, so the system is consistent for every right-hand side. Thus the contribution from the full-rank case equals
\begin{equation}\label{eq:full-rank-main}
\frac{m^r}{p^r}\cdot
\#\Bigl\{(S_1,\dots,S_r):\ \varnothing\ne S_i\subseteq [k],\ \text{all distinct},\ \rank(V)=r\Bigr\}.
\end{equation}
To estimate \eqref{eq:full-rank-main}, it suffices to show that almost all ordered \(r\)-tuples of distinct nonempty subsets of \([k]\) have full rank.

\begin{lemma}\label{lem:fullrank-dominates}
Uniformly for all \(1\le r\le R\), the number of ordered \(r\)-tuples of distinct nonempty subsets \((S_1,\dots,S_r)\) of \([k]\) with
\(\rank(V)<r\) is at most
\[
(M)_r\,p^{-\eta_1}
\]
for some absolute constant \(\eta_1>0\).
\end{lemma}

\begin{proof}
Let \(T_{r,d}\) be as in Lemma~\ref{lem:lowrankcount}. By the same argument as in the proof of Lemma~\ref{lem:lowrankcount}, but using only the trivial bound
\[
|W\cap\{0,1\}^r|\le 2^d,
\]
we obtain
\[
T_{r,d}\le 2^{r^2}(2^d)^k
\qquad (d<r).
\]
Hence
\[
\#\{(S_1,\dots,S_r):\ \text{all distinct and }\rank(V)<r\}
=
\sum_{d<r}T_{r,d}
\le
2^{r^2}\sum_{d<r}(2^d)^k
\le
r\,2^{r^2}(2^k)^{r-1}.
\]
On the other hand, since \(r\le R\) and \(M=2^k-1\asymp 2^k\), we have
\[
(M)_r\asymp (2^k)^r
\]
uniformly for \(r\le R\) (for instance by \eqref{eq:factorial_mr_v1} with \(m=1\)). Therefore
\[
\frac{\#\{(S_1,\dots,S_r):\ \text{all distinct and }\rank(V)<r\}}{(M)_r}
\ll
r\,2^{r^2}\,2^{-k}
=
p^{-1+o(1)}.
\]
Since \(r\le R\le (\log p)^\beta\), we have \(2^{r^2}=p^{o(1)}\), while
\(2^{-k}=p^{-1+o(1)}\). Thus the last quantity is at most \(p^{-1/2}\) for all sufficiently large \(p\). This proves the lemma.
\end{proof}

By Lemma~\ref{lem:fullrank-dominates}, the quantity in \eqref{eq:full-rank-main} equals
\[
\frac{m^r}{p^r}\Bigl((M)_r+O\bigl((M)_r p^{-\eta_1}\bigr)\Bigr)
=
\frac{m^r(M)_r}{p^r}(1+O(p^{-\eta_1})).
\]
Since \eqref{eq:factorial_mr_v1} (applied first with \(m=1\), and then with the fixed value \(m\in\{1,2\}\)) gives
\[
(M)_r=M^r\exp\Bigl(O\Bigl(\frac{r^2}{M}\Bigr)\Bigr)
\qquad\text{and}\qquad
(mM)_r=(mM)^r\exp\Bigl(O\Bigl(\frac{r^2}{M}\Bigr)\Bigr)
\]
uniformly for \(r\le R\), it follows that
\[
m^r(M)_r=(mM)_r\exp\Bigl(O\Bigl(\frac{r^2}{M}\Bigr)\Bigr).
\]
Now \(r\le R=(\log p)^\beta\) and \(M=2^k-1\asymp p(\log p)^\alpha\), so \(r^2/M=p^{-1+o(1)}\). Hence there exists an absolute constant \(\eta_1'>0\) such that \(m^r(M)_r=(mM)_r\bigl(1+O(p^{-\eta_1'})\bigr)\) uniformly for \(r\le R\), and so the contribution from the full-rank case is
\[
\frac{(mM)_r}{p^r}\bigl(1+O(p^{-\eta_1''})\bigr),
\]
uniformly for all \(1\le r\le R\), where \(\eta_1'':=\min\{\eta_1,\eta_1'\}>0\). Together with~\eqref{eq:moment-rank} and~\eqref{eq:lowrank-negl} this yields
\[
\mathbb E[(X_B)_r]=\frac{(mM)_r}{p^r}\bigl(1+O(p^{-\eta_2})\bigr)
\qquad (1\leq r\le R),
\]
uniformly for all $B\subseteq \F_p^\times$ with $|B|\in\{1,2\}$, where $\eta_2:=\min\{\eta/2,\eta_1''/2\}>0$. By~\eqref{eq:factorial_mr_v1}, we have
\[
(mM)_r=(mM)^r\exp\Bigl(O\Bigl(\frac{r^2}{M}\Bigr)\Bigr)
\]
uniformly for \(r\le R\). Since \(r^2/M=p^{-1+o(1)}\), there exists an absolute constant
\(\eta_3>0\) such that \((mM)_r=(mM)^r\bigl(1+O(p^{-\eta_3})\bigr)\) uniformly for \(r\le R\). As \(M/p=\lambda\), it follows that
\[
\frac{(mM)_r}{p^r}
=(m\lambda)^r\bigl(1+O(p^{-\eta_3})\bigr).
\]
Therefore
\begin{equation}\label{eq:factorial-moment-22}
\mathbb E[(X_B)_r]=(m\lambda)^r\bigl(1+O(p^{-\eta_4})\bigr)
\qquad (r\le R),
\end{equation}
uniformly for all \(B\subseteq \F_p^\times\) with \(|B|\in\{1,2\}\), where
\(\eta_4:=\min\{\eta_2,\eta_3\}>0\), proving \eqref{eq:goal-moment}.

Now, set $\widetilde R:=R-1$. Applying \eqref{eq:bonferroni} with truncation level $\widetilde R$ and $Z=X_B$ gives
\[
\mathbb P(X_B=0)
=\sum_{r=0}^{\widetilde R}(-1)^r\frac{\mathbb E[(X_B)_r]}{r!}
 + O\left(\frac{\mathbb E[(X_B)_{\widetilde R+1}]}{(\widetilde R+1)!}\right).
\]
Since $\widetilde R+1=R$, we may use \eqref{eq:factorial-moment-22} for every $r\le R$ to obtain, uniformly in $B$ with $|B|\in\{1,2\}$,
\[
\sum_{r=0}^{\widetilde R}(-1)^r\frac{\mathbb E[(X_B)_r]}{r!}
=\sum_{r=0}^{\widetilde R}(-1)^r\frac{(m\lambda)^r}{r!}
 + O\left(p^{-\eta_4}\sum_{r=0}^{\widetilde R}\frac{(m\lambda)^r}{r!}\right),
\]and by \eqref{eq:goal-moment} we have\[
\frac{\mathbb E[(X_B)_{\widetilde R+1}]}{(\widetilde R+1)!}
=\frac{\mathbb E[(X_B)_R]}{R!}
=\frac{(m\lambda)^R}{R!}\,(1+o(1)).
\]
Using $\sum_{r=0}^{\widetilde R}(m\lambda)^r/r!\le e^{m\lambda}$ and $e^{m\lambda}=p^{o(1)}$ (by Lemma~\ref{lem:lambda}), the error
$O(p^{-\eta_4}e^{m\lambda})$ is $o(e^{-m\lambda})$. Moreover, since $\lambda=o(R)$, Stirling's formula yields
\[
\frac{(m\lambda)^R}{R!} \le \left(\frac{e m\lambda}{R}\right)^R  =o(e^{-m\lambda}).
\]
Finally, because $m\lambda=o(\widetilde R)$ (by Lemma~\ref{lem:lambda}), the alternating Taylor remainder gives
\[
\sum_{r=0}^{\widetilde R}(-1)^r\frac{(m\lambda)^r}{r!}
= e^{-m\lambda} + O\left(\frac{(m\lambda)^{\widetilde R+1}}{(\widetilde R+1)!}\right)
= e^{-m\lambda}+o(e^{-m\lambda}).
\]
Combining the above estimates, we conclude
\[
\mathbb P(X_B=0)=e^{-m\lambda}\,(1+o(1)),
\]
uniformly for all $B\subseteq \F_p^\times$ with $|B|=m\in\{1,2\}$.
This proves Proposition~\ref{prop:poisson-miss}.\qed

\section{Concluding Remarks}\label{sec:concluding}

Our main theorem shows that, for primes $p$,
\begin{equation}\label{eq:quant_lower_bound_with_small_o}
f(p) \ge \log_2 p + \left(\frac{1}{2\log 2}+o(1)\right)\log\log p.
\end{equation}
This demonstrates that the upper bound $f(N) \le \log_2 N + o(\log\log N)$ cannot hold uniformly over all finite abelian groups of order $N$. Nevertheless, certain abelian groups $G$ of order $N$ do exist for which, with high probability, a uniformly random subset of size $\log_2 N + o(\log\log N)$ (or even smaller size) generates all elements of $G$. For a related discussion in the special case $G=(\mathbb Z/2\mathbb Z)^d$, see Sothanaphan~\cite{Sothanaphan26}.

Erd\H{o}s and R\'{e}nyi~\cite[Theorem~2]{ErRe65} proved that for all $N$,
\begin{equation}\label{eq:ER-upper}
f(N) \le \log_2 N + \frac{1}{\log 2}\log\log N + O(1).
\end{equation}
It is therefore natural to ask for an asymptotically sharp second-order term.

\begin{problem}
Determine the value of
\[
c_*:=\limsup_{N\to\infty}\frac{f(N)-\log_2 N}{\log\log N}.
\]
\end{problem}

Using \eqref{eq:quant_lower_bound_with_small_o} and \eqref{eq:ER-upper}, we can derive that  \(\frac{1}{2\log 2}\le c_* \le \frac{1}{\log 2}\). It is plausible that the upper bound is sharp, but the present method does not reach the regime \(c_*>\frac{1}{2\log 2}\). One evident bottleneck is the proof of Lemma~\ref{lem:lowrankcount}. It would be very interesting to push the lower bound constant beyond \(\frac{1}{2\log 2}\).

\bigskip

{\noindent \bf Acknowledgements.}
We thank Wouter van Doorn for helpful suggestions that improved the presentation of this paper. We are grateful to Nat Sothanaphan for providing relevant materials in~\cite{Sothanaphan26}. We also thank Boris Alexeev, Ingo Alth\"{o}fer, Thomas Bloom, Mehtaab Sawhney, Terence Tao, and Shengtong Zhang for helpful discussions on the Erd\H{o}s Problems website forum.
Finally, we are grateful to Thomas Bloom for founding and maintaining the Erd\H{o}s Problems website~\cite{EP543}. This work is supported by National Key Research and Development Program of China 2023YFA1010201, National Natural Science Foundation of China grant 12125106, and Innovation Program for Quantum Science and Technology 2021ZD0302902.

\medskip

{\noindent \bf Disclosure.}
An AI assistant was used as an exploratory aid at an early stage of this
project. All statements, proofs, and the final presentation in this paper were
independently verified and written by the authors. A fuller historical note on
the exploratory phase is recorded in the arXiv version of this paper~\cite[Section~1.1]{MaTa26}.

\end{document}